\documentclass{elsarticle}
\makeatletter
    \def\ps@pprintTitle{%
       \let\@oddhead\@empty
       \let\@evenhead\@empty
       \def\@oddfoot{\centerline{\thepage}}%
       \let\@evenfoot1
    }
\makeatother
\usepackage{amssymb}
\usepackage[utf8]{inputenc}
\usepackage[OT1]{fontenc}
\usepackage{comment} 
\usepackage{fullpage} 
\usepackage[english]{babel}
\usepackage{color}
\usepackage{amsthm,amssymb,bbm}
\usepackage{helvet}
\usepackage{hyperref}
\usepackage[ruled,algosection,noline]{algorithm2e}
\usepackage{mathdots}
\usepackage{MnSymbol}
\usepackage{graphicx}
\usepackage{subfig}
\graphicspath{ {fig/} }

\newtheorem{theorem}{Theorem}[section]
\newtheorem{proposition}[theorem]{Proposition}
\newtheorem{definition}[theorem]{Definition}
\newtheorem{lemma}[theorem]{Lemma}
\newtheorem{example}[theorem]{Example}

\newcommand{\tmmathbf}[1]{\mathbf{#1}}

\newcommand{\bpartial}{{\ensuremath{\boldsymbol{\partial}}}} 

\newcommand{\cA}{{\ensuremath{\mathcal{A}}}}

\newcommand{\cM}{{\ensuremath{\mathcal{M}}}}

\newcommand{\cV}{{\ensuremath{\mathcal{V}}}}

\newcommand{\CC}{\mathbb{C}}

\newcommand{\NN}{{\ensuremath{\mathbb{N}}}}

\DeclareMathOperator{\Hom}{Hom}
\DeclareMathOperator{\rank}{rank}
\DeclareMathOperator{\img}{img}
\DeclareMathOperator{\diag}{diag}

\def\PolExp{\mathcal{POLYEXP}}

\def\be{\mathbf{e}}

\def\bu{\mathbf{u}}
\def\bv{\mathbf{v}}

\def\bx{\mathbf{x}}
\def\by{\mathbf{y}}

\def\<{\langle}
\def\>{\rangle}

\def\Cx{\CC[\bx]}
\def\T{{\scriptscriptstyle\mathsf{T}}}
\def\H{{\scriptscriptstyle\mathsf{H}}}

\title{Structured low rank decomposition of multivariate Hankel matrices}
\author[inria,lama]{J. Harmouch}

\author[lama]{H. Khalil}

\author[inria]{B. Mourrain}
\ead{bernard.mourrain@inria.fr}

\address[inria]{UCA, Inria, \textsc{Aromath}, Sophia Antipolis, France}
\address[lama]{Laboratory of mathematics and its applications LaMa-Lebanon, Lebanese University}

\begin{document}
\begin{abstract}
We study the decomposition of a multivariate Hankel matrix $H_{\sigma}$ as a sum of Hankel matrices of small rank in correlation with the decomposition of its symbol $\sigma$ as a sum of polynomial-exponential series. 
We present a new algorithm to compute the low rank decomposition of the Hankel operator and the decomposition of its symbol exploiting the properties of the associated Artinian Gorenstein quotient algebra $\cA_{\sigma}$. 
A basis of $\cA_{\sigma}$ is computed from the Singular Value Decomposition of a sub-matrix of the Hankel matrix $H_{\sigma}$. The frequencies and the weights are deduced from the generalized eigenvectors of pencils of shifted sub-matrices of $H_{\sigma}$. Explicit formula for the weights in terms of the eigenvectors avoid us to solve a Vandermonde system. This new method is a multivariate generalization of the so-called Pencil method for solving Prony-type decomposition problems. We analyse its numerical behaviour in the presence of noisy input moments, and describe a rescaling technique which improves the numerical quality of the reconstruction for frequencies of high amplitudes. We also present a new Newton iteration, which converges locally to the closest multivariate Hankel matrix of low rank and show its impact for correcting errors on input moments.
\end{abstract}

\maketitle

\noindent{}{\bf AMS classification:} 14Q20,  68W30, 47B35, 15B05\\
\textbf{Keywords:} {Hankel; polynomial; exponential series; low rank decomposition; eigenvector; Singular Value Decomposition.}

\section{Introduction}
Structured matrices such as Toeplitz or Hankel matrices appear in many problems. 
They are naturally associated to operations on polynomials or series \cite{fuhrmann_polynomial_2012}.
The correlation with polynomial algebra can be exploited to accelerate matrix computations \cite{bini_polynomial_1994}.
The associated algebraic model provides useful information on the problem to be solved or the phenomena to be analysed. 
Understanding its structure often yields a better insight on the problem and its solution. 
In many cases, an efficient way to analyze the structure of the underlying models is to decompose the structured matrix into a sum of low rank matrices of the same structure. This low rank decomposition has applications in many domains \cite{markovsky_low_2012} and appears under different formulations \cite{landsberg_tensors:_2011,brachat_symmetric_2010,bernardi_general_2013}.

In this paper, we study specifically the class of Hankel matrices.
We investigate the problem of decomposing a Hankel matrix as a sum of indecomposable Hankel matrices of low rank. Natural questions arise.
What are the indecomposable Hankel matrices? Are they necessarily of rank $1$ ? 
How to compute a decomposition of a Hankel matrix as a sum of indecomposable Hankel matrices? Is the structured low rank decomposition of a Hankel matrix unique ? 

These questions have simple answers for non-structured or dense matrices: The indecomposable dense matrices are the matrices of rank one, which are the tensor product of two vectors. 
The Singular Value Decomposition of a dense matrix yields a decomposition as a minimal sum of rank one matrices, but this decomposition is not unique.

It turns out that for the Hankel structure, the answers to these questions are not so direct and involve the analysis of the so-called symbol associated to the Hankel matrix. The symbol is a formal power series defined from the coefficients of the Hankel matrix. As we will see, the structured decomposition of an Hankel matrix is closely related to the decomposition of the symbol as a sum of polynomial-exponential series.

The decomposition of the symbol of a Hankel matrix is a problem, which 
has a long history. The first work on this problem is probably due to  Gaspard-Clair-Fran{\c c}ois-Marie Riche de Prony {\cite{baron_de_prony_essai_1795}}. He proposed a method to reconstruct a sum of exponentials from the values at equally spaced data points, by computing a polynomial in the kernel of a Hankel matrix, and deducing the decomposition from the roots of this polynomial.
Since then, many works have been developed to address the decomposition problem in the univariate case, using linear algebra tools on Hankel matrices such as Pencil method \cite{hua_matrix_1990}, ESPRIT method {\cite{roy_esprit-estimation_1989}}
or MUSIC method {\cite{swindlehurst_performance_1992}}. 
Other methods such as {\cite{golub_separable_2003}}, 
approximate Prony Method \cite{beylkin_approximation_2005}, \cite{potts_nonlinear_2011} use minimization techniques, to recover the frequencies or the weights in the sum of exponential functions.
See {\cite{pereyra_exponential_2012}}[chap. 1] for a survey on some of these approaches. 

The numerical behavior of these methods has also been studied. The condition number of univariate Hankel matrices, which decomposition involves real points has been investigated in \cite{tyrtyshnikov_how_1994}, \cite{beckermann_condition_1997}. It is shown that it grows exponentially with the dimension of the matrix. The condition number of Vandermonde matrices of general complex points has been studied recently in \cite{pan_how_2016}.
In \cite{beckermann_numerical_2007}, the numerical sensitivity of the generalized eigenvalues of 
pencils of Hankel matrices appearing in Prony's method has been analysed.

The development of multivariate decomposition methods is more recent.
Extension of the univariate approaches have been considered
e.g. in \cite{andersson_nonlinear_2010}, {\cite{potts_parameter_2013}}, \cite{peter_pronys_2015}.
These methods project the problem in one dimension and solve several univariate 
decomposition problems to recover the multivariate decomposition by least square 
minimization from a grid of frequencies.
In \cite{peter_pronys_2015}, {\cite{kunis_multivariate_2016}}, the decomposition problem is transformed into the solution of an overdetermined polynomial system associated to the kernel of these Hankel matrices, and the frequencies of the exponential terms are found by general polynomial solver. 
These methods involves Hankel matrices of size exponential in the number of variables of the problem 
or moments of order at least twice the number of terms in the decomposition.
In {\cite{sauer_pronys_2016-1}}, an $H$-basis of the ideal defining the frequencies is computed 
from Hankel matrices built from moments of big enough order. Tables of multiplications are deduced from the $H$-basis and their eigenvalues yield the frequencies of the exponential terms. The weights are computed as the solution of a Vandermonde linear system. Moments of order bigger than twice the degree of an H-basis are involved in this construction.

\paragraph{Contributions}
We study the decomposition of a multivariate Hankel matrix as a sum of Hankel matrices of small rank in correlation with the decomposition of its symbol $\sigma$ as a sum of polynomial-exponential series. We show how to recover efficiently this decomposition from the structure of the quotient algebra $\cA_{\sigma}$ of polynomials modulo the kernel of the corresponding Hankel operator $H_{\sigma}$.   
In particular, a basis of $\cA_{\sigma}$ can be extracted from any maximal non-zero minor of the matrix of $H_{\sigma}$. We also show how to compute the matrix of multiplication by a variable in the basis of $\cA_{\sigma}$ 
from sub-matrices of the matrix of $H_{\sigma}$. We describe how the frequencies of the polynomial-exponential decomposition of the symbol can be deduced from the eigenvectors of these matrices.
Exploiting properties of these multiplication operators, we show that the weights of the decomposition can be recovered directly from these eigenvectors, avoiding the solution of a Vandermonde system.
We present a new algorithm to compute the low rank decomposition of $H_{\sigma}$ and the decomposition of its symbol as a sum of polynomial-exponential series from sub-matrices of the matrix of $H_{\sigma}$. A basis of $\cA_{\sigma}$ is computed from the Singular Value Decomposition of a sub-matrix. The frequencies and the weights are deduced from the generalized eigenvectors of pencils of sub-matrices of $H_{\sigma}$. This new method is a multivariate generalization of the so-called Pencil method for solving Prony-type decomposition problems. It can be used to decompose series as sums of polynomial-exponential functions from moments and provides structured low rank decomposition of multivariate Hankel matrices. 
We analyse its numerical behaviour in the presence of noisy input moments, for different numbers of variables, of exponential terms of the symbol and different amplitudes of the frequencies. 
We present a rescaling technique, which improves the numerical quality of the reconstruction for frequencies of high amplitudes. We also present a new Newton iteration, which converges locally to the multivariate Hankel matrix of a given rank the closest to a given input Hankel matrix. Numerical experimentations show that the Newton iteration combined the decomposition method allows to compute accurately and efficiently the polynomial-exponential decomposition of the symbol, even for noisy input moments.

\paragraph{Structure of the paper} The next section describes multivariate Hankel operators, their symbol and the generalization of Kronecker theorem, which establishes a correspondence between Hankel operators of finite rank and polynomial-exponential series. In Section \ref{sec:3}, we recall techniques exploiting the properties of multiplication operators for solving polynomial systems and show how they can be used for the Artinian Gorenstein algebra associated to Hankel operators of finite rank. In Section \ref{sec:4}, we describe in details the decomposition algorithm. Finally in section \ref{sec:5}, we present numerical experimentations, showing the numerical behaviour of the decomposition method for noisy input moments and the improvements obtained by  rescaling and by an iterative projection method.

\section{Hankel matrices and operators}
{\em Hankel matrices} are structured matrices of the form 
$$
H= [\sigma_{i+j}]_{0\le i\le l, 0\le j\le m}
$$
where the entry $\sigma_{i+j}$ of the $i^{\mathrm{th}}$ row and the  $j^{\mathrm{th}}$ columns depends only on the sum $i+j$.
By reversing the order of the columns or the rows, we obtain {\em Toeplitz matrices}, which entries depend on the difference of the row and column indices. Exploiting their structure leads to superfast methods for many linear algebra operations such as matrix-vector product, solution of linear systems, \ldots (see e.g. \cite{bini_polynomial_1994}). 

A Hankel matrix is a sub-matrix of the matrix of the {\em Hankel operator} associated to a sequence $\sigma=(\sigma_{k})\in \CC^{\NN}$: 
  \begin{eqnarray*}
    H_{\sigma}: L_{0}(\CC^{\NN}) & \rightarrow &\CC^{\NN} \\
    (p_{k})_{k} & \mapsto & (\sum_{k} p_{k} \sigma_{k+l})_{l\in \NN} 
 \end{eqnarray*}
where $L_{0}(\CC^{\NN})$ is the set of sequences  of $\CC^{\NN}$ with a finite support.

{\em Multivariate Hankel matrices} have a similar structure of the form
$$ 
H = [\sigma_{\alpha+\beta}]_{\alpha\in A, \beta\in B}
$$
where $A,B\subset \NN^{n}$ are subsets of multi-indices $\alpha=(\alpha_{1},\ldots,\alpha_{n})\in \NN^{n}$  $\beta=(\beta_{1},\ldots,\beta_{n})\in \NN^{n}$ indexing respectively the rows and columns.
{\em Multivariate Hankel operators} are associated to multi-index sequences $\sigma=(\sigma_{\alpha})_{\alpha\in \NN^{n}}\in \CC^{\NN^{n}}$:
\begin{eqnarray}\label{eq:hankel}
    H_{\sigma}: L_{0}(\CC^{\NN^{n}})& \rightarrow &\CC^{\NN^{n}} \\ 
    (p_{\alpha})_{\alpha} & \mapsto & (\sum_{\alpha} p_{\alpha} \nonumber \sigma_{\alpha+\beta})_{\beta\in \NN^{n}} 
 \end{eqnarray}
 where $L_{0}(\CC^{\NN^{n}})$ is the set of sequences  of $\CC^{\NN^{n}}$ with a finite support. In order to describe the algebraic properties of Hankel operators, we will identify hereafter the space $L_{0}(\CC^{\NN^{n}})$ with the ring $\CC[\bx]=\CC[x_1,\ldots, x_n]$ of polynomials in the variables $\bx=(x_1,\ldots, x_n)$ with coefficients in $\CC$.

The set of multi-index sequences $\CC^{\NN^{n}}$ can be identified with the ring of formal power series
 $\CC[[y_{1},\ldots,y_{n}]]=\CC[[\by]]$. A sequence $\sigma=(\sigma_{\alpha})_{\alpha}$ is identified with the series
$$ 
  \sigma (\by) = \sum_{\alpha \in \NN^n} \sigma_{\alpha} \frac{\by^{\alpha}}{\alpha !} \in  \CC[[\by]]
$$
where $\by^{\alpha}=y_{1}^{\alpha_{1}}\cdots y_{n}^{\alpha_{n}}$, $\alpha!= \prod_{i=1}^{n} \alpha_{i}!$ for $\alpha=(\alpha_{1},\ldots,\alpha_{n})\in \NN^{n}$.
It can also be interpreted as a linear functional on polynomials as follows:
$$
\begin{array}{rcl}
\sigma: \Cx & \rightarrow & \CC\\
p = \sum_{\alpha \in A\subset \NN^n } p_{\alpha}  \bx^{\alpha} & \mapsto& \langle \sigma
   \mid p \rangle = \sum_{\alpha \in A \subset \NN^n} p_{\alpha}   \sigma_{\alpha}.
\end{array}
$$
The identification of $\sigma$ with an element of $\Cx^{\ast}=\Hom_{\CC}(\CC [\bx], \CC)$ is uniquely defined by its
coefficients $\sigma_{\alpha} = \langle \sigma \mid
\bx^{\alpha} \rangle$ for $\alpha \in \NN^n$, which are
called the {\em{moments}} of $\sigma$.
This allows us to identify  the dual $\Cx^{\ast}$ with $\CC [[\by]]$ or with the set of multi-index sequences $\CC^{\NN^{n}}$.

The dual space $\Cx^{\ast} \equiv \CC [[\by]]$ has a natural structure of $\CC[\bx]$-module, defined as follows: 
\begin{eqnarray*}
\forall \sigma \in \CC [[\by]], \forall p , q \in \CC[\bx],\   \langle p  \star \sigma \mid q  \rangle  = \langle \sigma \mid p  q  \rangle.
\end{eqnarray*}
For a polynomial $p = \sum_{\alpha \in \NN^n} p_{\alpha} \bx^{\alpha}$ with $p_{\alpha}=0$ for almost all $\alpha \in \NN^n$, we have
$$ 
p\star \sigma = \sum _{\beta\in \NN^{n}}  (\sum_{\alpha\in \NN^{n}} p_{\alpha} \sigma_{\alpha+\beta}) \frac{\by^{\alpha}}{\alpha !}. 
$$
We check that $p\star \sigma= p(\partial_1, \ldots, \partial_n)(\sigma)=p(\bpartial)(\sigma)$ where $\bpartial=(\partial_1,\ldots, \partial_n)$ and $\partial_{i}$ is the derivation with respect to the variable $y_{i}$.

Identifying $L_{0}(\CC^{\NN^{n}})$ with $\CC[\bx]$, the Hankel operator \eqref{eq:hankel} is nothing else than the operator of multiplication by $\sigma$:
\begin{eqnarray*}
    H_{\sigma}: \CC[\bx]& \rightarrow &\CC [[\by]] \\ 
    p & \mapsto & p\star \sigma
\end{eqnarray*}
Truncated Hankel operators are obtained by restriction of Hankel operators.
For $A,B\subset \NN^{n}$, let $\<\bx^{B}\>\subset \CC[\bx]$, $\< \by^{A} \>\subset \CC[[\by]]$ be the vector spaces spanned respectively by the monomials $\bx^{\beta}$ for $\beta\in B$ and $\by^{\alpha}$ for $\alpha\in A$. The truncated Hankel operator of $\sigma$ on $A,B$ is 
$$  
\begin{array}{rcl}
  H^{A,B}_{\sigma} : \<\bx^{B}\> & \rightarrow & \<\by^{A}\>
  \\ 
 p = \sum_{\beta\in B} p_{\beta}\bx^{\beta}& \mapsto & \sum_{\alpha\in A} (\sum_{\beta\in B} p_{\alpha} \sigma_{\alpha+\beta}) \frac{\by^{\alpha}}{\alpha !} = p\star \sigma_{|\<\bx^{A}\>}
\end{array}
$$
The matrix of $H_{\sigma}^{A,B}$ in the bases $(\bx^{\beta})_{\beta\in B}$ and $(\frac{\by^{\alpha}}{\alpha!} )_{\alpha\in A}$ is of the form:
$$
H_{\sigma}^{A,B}=[\sigma_{\alpha+\beta}]_{\alpha\in A, \beta\in B}. 
$$
It is also called the {\em moment matrix} of $\sigma$. 
Multivariate Hankel matrices have a structure, which can be exploited to accelerate linear algebra operations (see e.g. \cite{mourrain_multivariate_2000} for more details).

\paragraph{Example} Consider the series $\sigma= 1 + 2 y_{1} + 3 y_{2} + 4 \frac{y_{1}^{2}}{2} + 5 y_{1} y_{2} + 6  \frac{y_{2}^{2}}{2} + 7  \frac{y_{1}^{3}}{6}
+ 8  \frac{y_{1}^{2} y_{2}}{2}  + \cdots \in \CC[[y_{1}, y_{2}]]$. Its truncated Hankel matrix on $A=[(0,0),(1,0), (0,1)]$ (corresponding to the monomials $1,x_{1}, x_{2}$), 
$B=[(0,0),(1,0),(0,1)$, $(2,0)]$ (corresponding to the monomials $1,x_{1},x_{2},x_{1}^{2}$) is 
$$ 
H_{\sigma}^{A,B} =
\left[
  \begin{array}{cccc}
    1 & 2 & 3 & 4\\ 
    2 & 4 & 5 & 7\\ 
    3 & 5 & 6 & 8 \\ 
\end{array}
 \right]
 $$

For $d\in \NN$, we denote by $\Cx_{d}$ the vector space of polynomials of degree $\le d$. Its dimension is $s_d={n+d \choose n}$.
For $d, d'\in \NN$, we denote by $H_{\sigma}^{d,d'}$ the Hankel matrix of $\sigma$ on the subset of monomials in $\bx$ respectively of degree $\le d$ and $\le d'$. We also denote by $H_{\sigma}^{d,d'}$ the corresponding truncated Hankel operator of $H_{\sigma}$ from $\Cx_{d'}$ to $(\Cx_{d})^{\ast}$.
More generally, for $U=\{u_1, \ldots, u_l\}\subset \Cx$, $V=\{v_1,\ldots, v_m\}\subset \Cx$, the Hankel matrix of $\sigma$ on $U$, $V$ is $H_{\sigma}^{U,V}= (\< \sigma\mid  u_{i}\, v_{j}\>)_{1\leq i\leq l, 1\leq j\leq m}$. We use the same notation $H_{\sigma}^{U,V}$ for the truncated Hankel operator from $\< V\>$ to $\< U\>^{\ast}$.

\subsection{Hankel operator of finite rank}
We are interested in structured decompositions of Hankel matrices (resp. operators) as sums of Hankel matrices (resp. operators) of low rank. 
This raises the question of describing the Hankel operators of finite rank and leads to the problem of decomposing them into indecomposable Hankel operators of low rank.

A first answer is given by the celebrated theorem of Kronecker \cite{kronecker_zur_1881}.
\begin{theorem}[Kronecker Theorem]
    The Hankel operator
    $$
    H_{\sigma}: (p_{k}) \in L_{0}(\CC^{\NN})\mapsto (\sum_{k} p_{k}\sigma_{k+l})_{l\in \NN}\in \CC^{\NN}
    $$
is of {finite rank} $r$, if and only if, there exist polynomials $\omega_{1},\ldots, \omega_{r'}\in \CC[y]$ and 
$\xi_{1},\ldots,\xi_{r'}\in \CC$ distinct s.t.
$$
\sigma_{n}= \sum_{i=1}^{r'} \omega_{i}(n)\xi_{i}^{n}
$$ with $\sum_{i=1}^{r'}(\deg(\omega_{i})+1)=r$.
\end{theorem}
This results says that the Hankel operator $H_{\sigma}$ is of finite rank, if and only if, its symbol $\sigma$ is of the form
$$ 
\sigma(y)= \sum_{n\in\NN} \sigma_n \frac{y^{n}}{n!} = \sum_{i=1}^{r'} \tilde{\omega}_{i}(y) e^{\xi_{i} y} 
$$
for some univariate polynomials $\tilde{\omega}_{i}(y)\in \CC[y]$ and distinct complex numbers $\xi_{i}$ $i=1,\ldots,r'$. Moreover, the rank of $H_{\sigma}$ is $r = \sum_{i=1}^{r'} (\deg(\tilde{\omega}_{i})+1)$.

The previous result admits a direct generalization to multivariate Hankel operators, using polynomial-exponential series. 
\begin{definition}
For $\xi=(\xi_{1},\ldots, \xi_{n})\in \CC^{n}$, we denote $\be_{\xi}(\by)=e^{y_{1}\xi_{1}+ \cdots+ y_{n}\xi_{n}} = \sum_{\alpha\in \NN^{n}}  \xi^{\alpha}\, \frac{\by^{\alpha}}{\alpha!} \in \CC[[\by]]$ where $\xi^{\alpha}=\xi_{1}^{\alpha_{1}} \cdots \xi_{n}^{\alpha_{n}}$ for $\alpha=(\alpha_{1}, \ldots, \alpha_{n}) \in \NN^{n}$.

  Let $\PolExp(\by) = \left\{ \sigma = \sum_{i = 1}^r \omega_i (\by) \be_{\xi_i} (\by) \in
  \CC [[\by]] \mid \xi_i \in \CC^n, \omega_i  (\by)\in \CC [\by] \right\}$ be the set of polynomial-exponential series. The polynomials $\omega_i (\by)$ are called the weights of $\sigma$
and $\xi_{i}$ the frequencies.

For $\omega  (\by)\in \CC [\by]$, we denote by $\mu(\omega)$ the dimension of the vector space spanned by $\omega$ and its derivatives $\partial^{\alpha}\omega(\by)=\partial_1^{\alpha_1} \cdots \partial_{n}^{\alpha_{n}}  \omega(\by)$ of any order for $\alpha=(\alpha_{1}, \ldots, \alpha_{n}) \in \NN^{n}$.
\end{definition}
The next theorem characterizes the multivariate Hankel operators of finite rank in terms of their symbol \cite{mourrain_polynomial-exponential_2016}:
\begin{theorem}[Generalized Kronecker Theorem]\label{thm:gorenstein} Let $\sigma (\by) \in \CC [[\by]]$. Then $\rank H_{\sigma} =r < \infty$, if
  and only if, $\sigma (\by) = \sum_{i = 1}^{r'}\omega_i (\by) \tmmathbf{e}_{{\xi_i}} (\by)  \in \PolExp(\by)$ 
with 
$\omega_i (\by) \in \CC[\by] \setminus \{0\}$ and $\xi_i \in \CC^n$ pairwise distinct, with $r = \sum_{i = 1}^{r'} {\mu} (\omega_i)$ where
  ${\mu} (w_i)$ is the dimension of the inverse system spanned by
  $\omega_i (\by)$ and all its derivatives
  $\tmmathbf{\partial}^{\alpha} \omega_i (\by)$ for
  $\alpha = (\alpha_1, \ldots, \alpha_n) \in \NN^n$.
\end{theorem}

\begin{example} \rm For $\xi\in \CC^{n}$, the series $\be_{\xi} (\by) = \sum_{\alpha \in \NN^{n}} \xi^{\alpha}\, \frac{\by^{\alpha}}{\alpha!}= e^{\by\cdot\xi}$ represents the linear functional corresponding to the evaluation at $\xi$:
$$ 
\forall p\in R, \< \be_{\xi} | p\> = \sum_{\alpha \in \NN^{n}} p_{\alpha}\, \xi^{\alpha} =  p (\xi). 
$$
The Hankel operator $H_{\be_{\xi}}: p \mapsto p\star \be_{\xi} = p(\xi) \be_{\xi}$ is of rank $1$, since its image is spanned by $\be_{\xi}$.
For $A,B \subset \NN^{n}$, the Hankel matrix of $\be_{\xi}$ is
$H_{\xi}^{A,B}=[\xi^{\beta+\alpha}]_{\beta\in B,\alpha\in A}.$ 
If $H_{\xi}^{A,B}\neq 0$, it is a matrix of rank $1$.
\end{example}
Hankel operators associated to evaluations $\be_{\xi}$ are of rank $1$. As shown in the next example, a Hankel operator of rank $>1$ is not necessarily the sum of such Hankel operators of rank $1$.
\begin{example}  For $n=1$ and $\sigma=y$, we check that $H_{y}$ is of rank $2$, but it cannot be decomposed as a sum of two rank-one Hankel operators. If $A=\{1,x,x^2\}$, we have
$$ 
H_{y}^{A,A} =
 \left[
   \begin{array}{ccc}
0&     1 & 0\\ 
     1&    0 & 0 \\
     0&    0 & 0 \\
\end{array}
\right] 
\neq
\lambda_{1}\,
 \left[
   \begin{array}{ccc}
     1 & \xi_{1} & \xi_{1}^{2}\\ 
     \xi_{1} & \xi_{1}^{2} & \xi_{1}^{3}\\ 
    \xi_{1}^{2} & \xi_{1}^{3} &\xi_{1}^{4}\\     
   \end{array}
\right]
+
\lambda_{2}\,
 \left[
   \begin{array}{ccc}
     1 & \xi_{2} & \xi_{2}^{2}\\ 
     \xi_{2} & \xi_{2}^{2} & \xi_{2}^{3}\\ 
    \xi_{2}^{2} & \xi_{2}^{3} &\xi_{2}^{4}\\     
   \end{array}
\right]
$$
for $\lambda_1,\lambda_2, \xi_1,\xi_2 \in \CC$. This shows that the symbol $y$ is indecomposable as a sum of polynomial-exponential series, though it defines an Hankel operator of rank $2$.
\end{example}

\begin{definition} For $\sigma \in \CC[[\by]]$, we say that $\sigma$ is indecomposable if ${\sigma}$ cannot be written as a sum ${\sigma}= {\sigma_1} + {\sigma_{2}}$ with $\img H_{\sigma} = \img H_{\sigma_{1}} \oplus  \img H_{\sigma_{2}}$.
\end{definition}
\begin{proposition} The series $\omega(\by) \,\be_{\xi}(\by)$ with $\omega(\by) \in \CC[\by] \setminus \{0\}$ and $\xi \in \CC^n$ is indecomposable. 
\end{proposition}
\begin{proof}
Let $\sigma=\omega\, \be_{\xi}$ and $r=\mu(\omega)$ the rank of $H_{\sigma}$.
Suppose that $\sigma = \sigma_{1} + \sigma_{2}$ with $\img H_{\sigma} = \img H_{\sigma_{1}} \oplus  \img H_{\sigma_{2}}$. We assume that the rank of $H_{\sigma_{1}}$ is minimal.
By the Generalized Kronecker Theorem, $\sigma_{1}=\sum_{i=1}^{r_1} \omega_{1,i}\, \be_{\xi_{1,i}}$, $\sigma_{2}=\sum_{i=1}^{r_2} \omega_{2,i}\, \be_{\xi_{2,i}}$ with 
$\omega_{l,i} \in \Cx$, $\xi_{l,i}\in \CC^{n}$ and 
$$
\omega\, \be_{\xi} =  \sum_{i=1}^{r_1} \omega_{1,i}\, \be_{\xi_{1,i}} +
\sum_{i=1}^{r_2} \omega_{2,i}\, \be_{\xi_{2,i}}.
$$
By the independence of the polynomial-exponential series \cite{mourrain_polynomial-exponential_2016}[Lem. 2.7], we can assume that 
$\xi_{1,1}=\xi_{2,1}=\xi$ and that $\omega=\omega_{1,1}+\omega_{2,1}$ (possibly with $\omega_{2,1}=0$) and that  
$\omega_{1,i}=-\omega_{2,i}$ for $i=2,\ldots, r_{1}=r_{2}$ .
As $\rank H_{\sigma_{1}}= \sum_{i=1}^{r_{1}}\mu(\omega_{1,i})$ is minimal, we can assume moreover that $\omega_{1,i}=0$ for $i=2,\ldots, r_{1}$, that is, $r_{1}=r_{2}=1$.
Then, we have $\sigma=\omega\, \be_{\xi}$, $\sigma_{1}= \omega_{1} \,\be_{\xi}$
$\sigma_{2}= \omega_{2}\, \be_{\xi}$ with $\omega = \omega_{1} + \omega_{2}$. 
As $\img H_{\sigma_{i}} = \< \partial^{\alpha}(\omega_i)\, \be_{\xi} \>$, $i=1,2$, we have $\img H_{\sigma_{1}} \cap  \img H_{\sigma_{2}} \ni \be_{\xi}$ and $\img H_{\sigma}$ is not the direct sum of $\img H_{\sigma_{1}}$ and $\img H_{\sigma_{2}}$. This shows that $\sigma$ is indecomposable.
\end{proof}
The goal of this paper is to present a method to decompose the symbol of a Hankel operator as a sum of indecomposable polynomial-exponential series from truncated Hankel matrices.

\section{Structured decomposition of Hankel matrices}\label{sec:3}
In this section, we show how the decomposition of the symbol $\sigma$ of a Hankel operator $H_{\sigma}$ as a sum of polynomial-exponential series
reduces to the solution of polynomial equations. 
This corresponds to the decomposition of $H_{\sigma}$ as a sum of Hankel matrices of low rank.
We first recall classical techniques for solving polynomial systems and show how these 
methods can be applied on the Hankel matrix $H_{\sigma}$, to compute the decomposition.

\subsection{Solving polynomial equations by eigenvector computation}

A quotient algebra $\cA=\Cx/I$ is {\em Artinian} if it is of finite dimension over $\CC$. In this case, the ideal $I$ defines a finite number of roots $\cV(I)=\{\xi_{1},\ldots, \xi_{r'}\}=\{\xi \in \CC^{n}\mid \forall p \in I, p(\xi)=0 \}$ and we have a decomposition of $\cA$ as a sum of sub-algebras:
$$ 
\cA =\Cx/I = \cA_{1} \oplus \cdots \oplus \cA_{r'}
$$
where $\cA_{i}= \bu_{\xi_{i}} \cA \sim \Cx/Q_{i}$ and $Q_{i}$ is the primary component of $I$ associated to the root $\xi_{i}\in \CC^{n}$. The elements $\bu_{1},\ldots,\bu_{r'}$ satisfy the relations
$$ 
  {{\bu}}_{{\xi}_i}^{2} ({\bx})
  \equiv {\bu}_{{\xi}_i} ({\bx}), \ 
\sum_{i = 1}^r {\bu}_{{\xi}_i} (\bx)
  \equiv 1.
$$
The polynomials $\bu_{\xi_1}, \ldots, \bu_{\xi_{r'}}$ are called {{\em  idempotents } of $\cA$. The dimension of $\cA_{i}$ is the {\em multiplicity} of the point $\xi_{i}$. For more details, see {\cite{elkadi_introduction_2007}[Chap. 4]}.

For $g\in \Cx$, the multiplication operator
  $\cM_g$ is defined by
  \[ \begin{array}{cccl}
       \mathcal{M}_g : & \mathcal{A} & \to & \mathcal{A}\\
       & h & \mapsto & \mathcal{M}_g (h) = g\,h.
     \end{array} \]
The transpose $\cM_g^\T$ of the multiplication operator $\cM_g$ is
  \[ \begin{array}{cccl}
       \cM_g^\T : & \mathcal{A}^{\ast} & \to & \mathcal{A}^{\ast}\\
       & \Lambda & \mapsto & \cM_g^\T (\Lambda) = \Lambda \circ
       \cM_g = g \star \Lambda .
     \end{array} \]
The main property that we will use to recover the roots is the following \cite{elkadi_introduction_2007}[Thm. 4.23]:
\begin{proposition}\label{prop:eigen} Let $I$ be an ideal of $\Cx$ and
  suppose that $\cV(I) = \{ {\xi}_1, {\xi}_2, \ldots, {\xi}_{r'} \}$. Then
  \begin{itemize}
    \item for all $g \in \cA$, the eigenvalues of $\cM_g$ and
    $\cM_g^{\T}$ are the values $g (\mathbf{\xi}_1), \ldots, g
    ({\xi}_{r'})$ of the polynomial $g$ at the roots with multiplicities
    ${\mu}_i = \dim \cA_{i}$.
    
    \item The eigenvectors common to all $\cM_g^{\T}$ with $g \in
    \cA$ are - up to a scalar - the evaluations
    ${\be}_{{\xi}_1}, \ldots, {\be}_{{\xi}_{r'}}$.
  \end{itemize}
\end{proposition}

  If $B = \{ b_1, \ldots, b_r \}$ is a basis of $\cA$, then the coefficient vector of the evaluation $\be_{{\xi}_i}$ in the dual basis of $B$ is $\left[ \left\langle \be_{{\xi_i}}  |   b_j \right\rangle
  \right]_{\beta \in B} = [b_j ({\xi}_i)]_{i = 1 \ldots r} = B
  (\xi_i)$. The previous proposition says that if $M_g$ is the matrix of
  $\mathcal{M}_g$ in the basis $B$ of $\cA$, then
  \[ M_g^{\T} \,B ({\xi}_i) = g (\xi_i) \,B ({\xi}_i) . \]
  If moreover the basis $B$ contains the monomials $1, x_1, x_2, \ldots, x_n$,
  then the common eigenvectors of $M_g^{\T}$ are of the form
  ${\bv}_i = c\,  [1, \xi_{i, 1}, \ldots, \xi_{i, n}, \ldots]$ and the root $\xi_i$ can be
  computed from the coefficients of ${\bv}_i$ by taking the ratio of
  the coefficients of the monomials $x_1, \ldots, x_n$ by the
  coefficient of $1$: $\xi_{i, k} = \frac{{\bv}_{i, k + 1}}{{\bv}_{i, 1}}$.
  Thus computing the common eigenvectors of all the
  matrices $M_g^{\T}$ for $g \in \cA$ yield the roots ${\xi}_i$  ($i = 1, \ldots, r$).

  In practice, it is enough to compute the common eigenvectors of $M_{x_1}^{\T}, \ldots, M_{x_n}^{\T}$, since $\forall g \in \CC[\bx], M_{g}^{\T} = g (M_{x_1}^{\T},  \ldots, M_{x_n}^{\T})$. Therefore, the common eigenvectors $M_{x_1}^{\T}, \ldots, M_{x_n}^{\T}$ are also eigenvectors of any $M_{g}^{\T}$.
  
  The multiplicity structure, that is the dual $Q_i^{\perp}$ of each primary component $Q_{i}$ of $I$, also called the {\em inverse system} of the point $\xi_i$ can be deduced by linear algebra tools (see e.g. \cite{mourrain_isolated_1996}). 

In the case of simple roots, we have the following property \cite{elkadi_introduction_2007}[Chap. 4]:
\begin{proposition} \label{prop:simpleroot} If the roots $\{ {\xi}_1, {\xi}_2, \ldots, {\xi}_r \}$ of $I$ are simple (i.e. $\mu_i=\dim \cA_{i}=1$) then we have the following:
  \begin{itemize}
  \item $\mathbbm{u}=\{\bu_{\xi_1},\ldots, \bu_{\xi_r}\}$ is a basis of $\cA$.
  \item The polynomials ${\bu}_{\xi_1}, \ldots, \bu_{\xi_r}$ are interpolation polynomials at the roots $\xi_{i}$: $\bu_{\xi_i}(\xi_{j})=1$ if $i=j$ and $0$ otherwise. 
  \item The matrix of $\cM_{g}$ in the basis $\mathbbm{u}$ is the diagonal matrix $\diag(g(\xi_{1}),\ldots, g(\xi_{r}))$.
\end{itemize}
\end{proposition}
This proposition tells us that if $g$ is separating the roots, i.e. $g(\xi_{i})\neq g(\xi_{j})$ for $i\neq j$, then the eigenvectors of $\cM_{g}$ are, up to a scalar, interpolation polynomials at the roots.

\subsection{Artinian Gorenstein algebra of a multivariate Hankel operator}
We associate to a Hankel operator $H_{\sigma}$, the quotient $\cA_{\sigma}=\Cx/I_{\sigma}$ of the polynomial ring $\Cx$ modulo the kernel $I_{\sigma}=\{ p\in \Cx \mid \forall q \in R, \< \sigma\mid p q\>=0 \}$ of $H_{\sigma}$. We check that $I_{\sigma}$ is an ideal of $\Cx$, so that $\cA_{\sigma}$ is an algebra.

As $\cA_{\sigma}= \Cx/I_{\sigma}\sim \img H_{\sigma}$, the operator $H_{\sigma}$ is of finite rank $r$, if and only if, $\cA_{\sigma}$ is Artinian of dimension $\dim_{\CC} \cA_{\sigma}=r$ .

A quotient algebra $\cA$ is called {\em Gorenstein} if its dual $\cA^{*}=\Hom_{\CC}(\cA,\CC)$ is a free $\cA$-module generated by one element. 

In our context, we have the following equivalent properties \cite{mourrain_polynomial-exponential_2016}:
\begin{itemize}
 \item $\sigma=\sum_{i=1}^{r'} \omega_{i}(\by) \be_{\xi_{i}}(\by)$ with $\omega_{i}\in \CC[\by]$, $\xi_{i}\in \CC^{n}$ and $\sum_{i=1}^{r'}\mu(\omega_{i})=r$,
 \item $H_{\sigma}$ is of rank $r$,
 \item $\cA_{\sigma}$ is an Artinian Gorenstein algebra of dimension $r$.
\end{itemize}
 
The following proposition shows that the frequencies $\xi_{i}$ and the weights $\omega_{i}$ can be recovered from the ideal $I_{\sigma}$ (see \cite{mourrain_polynomial-exponential_2016} for more details):
\begin{proposition}\label{prop:iso}
  If $\sigma (\by) = \sum_{i = 1}^{r'} \omega_i (\by) \be_{{\xi_i}}(\by)$ with 
$\omega_i (\by) \in \CC[\by] \setminus \{0\}$ and $\xi_i \in \CC^n$ pairwise
  distinct, then we have the following properties:
  \begin{itemize}
    \item The points ${\xi}_1, {\xi}_2, \ldots, {\xi}_{r'} \in \CC^n$ are the common roots of the polynomials in $I_{\sigma} = \ker H_{\sigma} = \{ p \in \Cx \mid \forall q \in \Cx, \langle \sigma |   p q \rangle = 0 \}$.
    
    \item The series $\omega_i (\by) \be_{\xi_i}$ is a
    generator of the inverse system of $Q_i^{\bot}$, where $Q_i$ is the
    primary component of $I_{\sigma}$ associated to $\xi_i$ such that $\dim \Cx/Q_i=\mu(\omega_i)$.
 \end{itemize}
\end{proposition}
This result tells us that the problem of decomposing $\sigma$ as a sum of polynomial-exponential series reduces to  the solution of the polynomial equations $p=0$ for $p$ in the kernel $I_{\sigma}$ of $H_{\sigma}$.

Another property that will be helpful to determine a basis of $\cA_{\sigma}$ is the following:
\begin{lemma}\label{lem:basis} Let $B = \{ b_1, \ldots, b_r \}$, $B' = \{ b_1', \ldots,
  b_r' \} \subset \Cx$. If the matrix $H_{\sigma}^{B,B'} = (\langle \sigma |   b_i b_j' \rangle)_{1 \leq i,j\leq r}$ is invertible, then B and $B'$ are linearly independent in  $\cA_{\sigma}$.
\end{lemma}
\begin{proof}
  Suppose that $H_{\sigma}^{B, B'}$ is invertible. If there exists $p = \sum_i
  \lambda_i b_i$ ($\lambda_i \in \CC$) such that $p \equiv 0$ in
  $\cA_{\sigma}$. Then $p \star \sigma = 0$ and $\forall q \in R$,
  $\langle \sigma |   p q \rangle = 0$. In particular, for $j = 1, \ldots, r$ we have
  \[ \sum_{i = 1}^r \langle \sigma |   b_i b_j' \rangle \lambda_i = 0. \]
  As $H_{\sigma}^{B, B'}$ is invertible, $\lambda_i = 0$ for $i = 1, \ldots, r$ and
  $B$ is a family of linearly independent elements in $\cA_{\sigma}$.
  Since we have $(H_{\sigma}^{B, B'})^{\T} = H_{\sigma}^{B', B}$, \ we prove by a
  similar argument that $H_{\sigma}^{B, B'}$ invertible also implies that
  $B'$ is linearly independent in $\cA_{\sigma}$.
\end{proof}

By this Lemma, bases of $\cA_{\sigma}$ can be computed by identifying non-zero minors of maximal size of the matrix of $H_{\sigma}$.

\begin{proposition}\label{eq:HankelMult}Let $B,B'$ be basis of $\cA_{\sigma}$ and $g \in \Cx$. We have
  \begin{equation}
    {H}_{g \star \sigma}^{B,B'} = (M_{g}^{B})^{\T} H_{\sigma}^{B,B'}
    = H_{\sigma}^{B,B'} M_{g}^{B'}. \label{eq:Hankel}
  \end{equation}
where $M_g^B$ (resp. $M_g^{B'}$) is the matrix of the multiplication by $g$ in the basis $B$ (resp. $B'$) of $\cA
_{\sigma}$.
\end{proposition}
\begin{proof}
Let $B=\{b_1,\ldots,b_r\}, B'=\{b'_1,\ldots,b'_r\}$ be two bases of $\cA_{\sigma}$. 
We have $g\, {b}_j=\sum_{i=1}^{r} m_{i,j} {b}'_i + \kappa$ where $m_{i,j}$ is the $(i,j)$ entry of the matrix $M_g^{B}$ of multiplication by $g$ in the basis $B$ and $\kappa \in I_{\sigma}$.
Then,
\begin{align*}
({H}_{g \star \sigma}^{B,B'})_{[i,j]}
= \< \sigma \mid  g\, b_i\, b'_j \>
&
=  \< \sigma \mid  \sum_{l=1}^r m_{l,i}  b_l b'_j \> 
+  \< \sigma \mid  \kappa\, b_j\>  
=  \sum_{l=1}^r m'_{l,i} \< \sigma \mid  b_l b'_j \> 
= ((M_{g}^{B})^{t} {H}_{\sigma}^{B,B'})_{[i,j]}.
\end{align*}
Similarly, we have $g\, b'_j=\sum_{i=1}^{r} m'_{i,j} b'_i+ \kappa'$ where $m'_{i,j}$ is the $(i,j)$ entry of the matrix $M_g^{B'}$ of multiplication by $g$ in the basis $B'$ and $\kappa' \in I_{\sigma}$.
For ${1\leq i,j\leq r}$, the entry $(i,j)$ of ${H}_{g \star \sigma}^{B,B'}$ is 
\begin{align*}
({H}_{g \star \sigma}^{B,B'})_{[i,j]}
= \< \sigma \mid b_i \, g\, b'_j \>
&
=  \< \sigma \mid  \sum_{l=1}^r m_{l,j} b_i\, b'_l  \> +  \< \sigma \mid  b_i \kappa' \>  
=  \sum_{l=1}^r  \< \sigma \mid  b_i\, b'_l \>\,  m_{l,j}
= ({H}_{\sigma}^{B,B'} M_{g}^{B})_{[i,j]}.
\end{align*}
This concludes the proof of the relations (\ref{eq:Hankel}).
\end{proof}

We deduce the following property:
\begin{proposition}\label{decomposviamulti}
  \label{prop:geneigen} Let $\sigma (\by) = \sum_{i = 1}^r \omega_i (\by) \be_{\xi_i}(\by)$ with $\omega_i \in \CC[\by] \setminus \{0\}$ and $\xi_i \in \CC^n$ distinct and let $B,B'$ be bases of $\cA_{\sigma}$. We have the following properties:
  \begin{itemize}
    \item For $g\in \Cx$, $M_{g}^{B'}= (H_{\sigma}^{B,B'})^{-1} H_{g\star \sigma}^{B,B'}$, $(M_{g}^{B})^{\T}= H_{g\star \sigma}^{B,B'} (H_{\sigma}^{B,B'})^{-1}$.
    
    \item For $g \in \Cx$, the generalized eigenvalues of
    $(H_{g\star \sigma}^{B,B'}, H_{\sigma}^{B,B'})$ are $g({\xi}_i)$ with multiplicity ${\mu}_i ={\mu} (\omega_i)$,
    $i = 1, \ldots, r$.
    
    \item The generalized eigenvectors common to all $(H_{g\star\sigma}^{B,B'}, H_{\sigma}^{B,B'})$ for $g \in \Cx$ are - up to a
    scalar - $(H_{\sigma}^{B,B'})^{-1} \, B(\xi_i)$, $i=1,\ldots,r$.
  \end{itemize}
\end{proposition}
\begin{proof}
The two first points are direct consequences of  Propositions \ref{eq:HankelMult} and \ref{prop:eigen}. The third point is also a consequence of Proposition \ref{prop:eigen}, since the coordinate vector of the evaluation $\be_{\xi_i}$ in the dual basis of $B$ is $B(\xi_i)$ for $i=1,\ldots,r$.
\end{proof}

This proposition shows that the matrices of multiplication by an element $g$ in $\cA$, and thus the roots $\{\xi_1, \ldots, \xi_r\}$ and their multiplicity structure, can be computed from truncated Hankel matrices, provided we can determine bases $B$, $B'$ of $\cA_{\sigma}$. In practice, it is enough to compute the generalized eigenvectors common to $(H_{x_i\star\sigma}^{B,B'}, H_{\sigma}^{B,B'})$ for $i=1,\ldots, n$ to recover the roots.
As $H_{x_i\star\sigma}^{B,B'}= H_{\sigma}^{x_{i} B,B'}= H_{\sigma}^{B, x_{i}\, B'}$, the decomposition can be computed from sub-matrices of $H_{\sigma}^{B,B'^{+}}$  
or $H_{\sigma}^{B^{+},B'}$ where $B^{+}= B \cup x_1 B \cup \cdots \cup x_n B$, $B'^{+}= B' \cup x_1 B' \cup \cdots \cup x_n B'$. 
\section{Decomposition algorithm}\label{sec:4}
We are given the first moments $\sigma_{\alpha}, |\alpha| \leq d$ of the series $\sigma(\by)=\sum_{i=1}^{r}\omega_{i}\be_{\xi_{i}}(\by)$ with $\omega_{i} \in \CC\backslash{(0)}$ and $\xi_{i} \in \CC^{n}.$ The goal is to recover the number of terms r, the constant weights $\omega_{i}$ and the frequencies $\xi_{i}$ of the series $\sigma(\by)$.

\subsection{Computation of a basis}
The first problem is to find automatically bases $B_{1}$ and $B_{2}$ of the quotient algebra $\cA_{\sigma}$, of maximal sizes such that $H_{\sigma}^{B_{1},B_{2}}$ is invertible. Using Proposition \ref{decomposviamulti}, we will compute  the multiplication matrices $M_{g}^{B_{2}}$ for $g=x_{i},i=1,\ldots,n$. The frequencies $\xi_{j}$ and the weights $\omega_{j}$, $j=1,\ldots,r$ will be deduced from their eigenvectors, as described in section \ref{weights}.

Given the set of moments $(\sigma_{\alpha})_{|\alpha|\le d}$, we create two sets $A_{1}=(\bx^{\alpha})_{|\alpha| \leq d_1}$ and $A_{2}=(\bx^{\beta})_{|\beta| \leq d_2}$ of monomials such that $\alpha$ and $\beta$ are multi-indices in $\NN^{n}$
with $|\alpha|\le d_1$ and $|\beta|\le d_2$.
The degrees $d_{1}$ and $d_{2}$ are chosen such that $d_1+d_2< d$. Let $N_{1}=|A_{1}|$ and $N_{2}=|A_{2}|$. The truncated Hankel operator associated to $\sigma$ is:
$$\begin{array}{ccccc}
H_{\sigma}^{d_{1},d_{2}} & : & {\Cx}_{d_2} & \to & ({\Cx}_{d_1})^{*} \\
& & p & \mapsto & p\star\sigma
\end{array}$$
The Hankel matrix in these two monomial sets $A_{1}$ and $A_{2}$ is defined by $H_\sigma^{d_1,d_2}=[\sigma_{(\alpha+\beta)}]_{\substack{\mid \alpha \mid \leq d_1 \\ \mid \beta \mid \leq d_2}}$.

Computing the singular value decomposition of $H_{\sigma}^{d_1,d_2}$, we obtain 
$$
 H_{\sigma}^{d_1,d_2}=USV^{\T}
$$ 
 where $S$ is the diagonal matrix of all singular values of $H_{\sigma}^{d_1,d_2}$ arranged in a decreasing order, $U$ is an unitary matrix whose columns are the left singular vectors of $H_{\sigma}^{d_1,d_2}$, $V$ is an unitary matrix whose columns are the right singular vectors of $H_{\sigma}^{d_1,d_2}$. We denote by  $U^{\H}$ the hermitian transpose of $U$ and $\overline{V}$ the conjugate 
 of $V$.

Let $u_{i}=[u_{\alpha,i}]_{\alpha \in A_{1}}$ and $v_{j}=[v_{\beta,j}]_{\beta \in A_{2}}$  be  respectively the $i^{\textrm{th}}$ and $j^{\textrm{th}}$ columns of $U^{\H}$ and $\overline{V}$. They are vectors respectively in $\mathbb{C}^{N_{1}}$ and $\CC^{N_{2}}$. We denote by $u_{i}(\bx)= u_{i}^{\T}\, A_{1} =\sum_{|\alpha|\leq d_{1}} u_{\alpha,i} \bx^{\alpha}$ and $v_{j}(\bx)=v_{j}^{\T} \, A_{2}=\sum_{|\beta|\leq d_{2}} v_{\beta,j} \bx^{\beta}$ the corresponding polynomials. 
The bases formed by these first $r$ polynomials are denoted 
$U_{r}^{\H}:=(u_{i}(\bx))_{i=1,\ldots,r}$ and $\overline{V}_{r}:=(v_{j}(\bx))_{j=1,\ldots,r}$. We will also denote by $U_{r}^{\H}$ (resp. $\overline{V}_{r}$) the corresponding coefficient matrix, formed by the first rows (resp. columns) of $U^\H$ (resp. $\overline{V}$). We denote by $S_r$ the diagonal matrix of the first $r$ rows and columns of $S$, formed by the first $r$ singular values.
\begin{proposition} Let $\sigma=\sum_{i=1}^{r'}\omega_i(\by) \be_{\xi_{i}}$ with $\omega_{i}\in \CC[\by]$, $\xi_{i}\in \CC^n$ and $\sum_{i=1}^{r'}\mu(\omega_i)=r$. If $\rank H_{\sigma}^{d_1,d_2}=r$, then the sets of polynomials $U_{r}^{\H}$ and $\overline{V}_{r}$ are bases of $\cA_{\sigma}$. 
The matrix $M_{x_{i}}^{\overline{V}_{r}}$ associated to the multiplication operator by $x_{i}$ in the basis $\overline{V}_{r}$ of $\cA_{\sigma}$ is $M_{x_{i}}^{\overline{V}_{r}}= S_{r}^{-1}\,U_{r}^{\H}\,H_{x_{i} \star \sigma}^{d_1,d_2}\,\overline{V}_{r}$ $i=1,\ldots,n$.
\end{proposition}
\begin{proof}
The $(i,j)$ entry of the matrix  $H_\sigma^{U_{r}^{\H},\overline{V}_{r}}$ of the truncated Hankel operator of $\sigma$ with respect to $U_{r}^{\H}$ and $\overline{V}_{r}$ is equal to: 
\begin{equation} \label{eq1}
\begin{split} 
(H_\sigma^{U_{r}^{\H},\overline{V}_{r}})_{[i,j]} & = \< \sigma | u_{i}(\bx)v_{j}(\bx) \>\\
& =  \< \sigma | (\sum_{|\alpha| \leq d_{1}} u_{\alpha,i}\bx^{\alpha})\,(\sum_{|\beta| \leq d_{2}} v_{\beta,j} \bx^{\beta}) \> = \sum_{|\alpha| \leq d_{1}}u_{\alpha,i}  \sum_{|\beta| \leq d_{2}} \< \sigma | \bx^{\alpha} \bx^{\beta} \> v_{\beta,j}\\ & =[U_{r}^{\H}H_{\sigma}^{d_{1},d_{2}}\overline{V}_{r}]_{[i,j]}. 
\end{split}
\end{equation}
Using the SVD decomposition of $H_{\sigma}^{d_1,d_2}$, we have 
$$
 H_\sigma^{U_{r}^{\H},\overline{V}_{r}}=U_{r}^{\H}H_{\sigma}^{d_{1},d_{2}}\overline{V}_{r}=U_{r}^{\H}USV^{\T}\overline{V}_{r}=S_{r},
$$  
since $U^H\, U=\mathrm{Id}_{N_1}$, $V^T\, \overline{V}=\mathrm{Id}_{N_2}$.
As $r=\rank H_{\sigma}^{d_1,d_2}$, $S_{r}$ is invertible and by  Lemma \ref{lem:basis}, 
$U_{r}^{\H}$ and $\overline{V}_{r}$ are linearly independent in $\cA_{\sigma}$, which is a vector space of dimension $r$. Thus they are bases of $\cA_{\sigma}$. 

Let $H_{{x_{i}}\star\sigma}^{U_{r}^{\H},\overline{V}_{r}}$ be the matrix of the truncated Hankel operator of ${x_{i}}\star\sigma$ on the two bases $U_{r}^{\H}$ and $\overline{V}_{r}$. 
A similar computation yields $H_{{x_{i}}\star\sigma}^{U_{r}^{\H},\overline{V}_{r}}=U_{r}^{\H}H_{x_{i}\star \sigma}^{d_1,d_2}\overline{V}_{r}$, where $H_{x_{i}\star\sigma}^{d_1,d_2}$ is the matrix of the truncated Hankel operator of ${x}_{i}\star\sigma$ in the bases $A_{1}$ and $A_{2}$ for all $i=1,\dots,n$. Since $S_{r}$ is an invertible matrix, by Proposition \ref{decomposviamulti} we obtain $M_{\bx_{i}}^{\overline{V}_{r}}=(H_{\sigma}^{U_{r}^{\H},\overline{V}_{r}})^{-1} H_{x_{i}\star\sigma}^{U_{r}^{\H},\overline{V}_{r}}=S_{r}^{-1}U_{r}^{\H}H_{x_{i}\star\sigma}^{d_{1},d_{2}}\overline{V}_{r}$.	
\end{proof}	
By this proposition $U^{\H}_{r}$ and $\overline{V}_{r}$) are bases of 
$\cA_{\sigma}$.
By Proposition \ref{decomposviamulti}, the eigenvalues of $M_{x_{i}}^{\overline{V}_{r}}$	are  the $i^\text{th}$ coordinates $x_{i}(\xi_{j})=\xi_{j,i}$ of the roots $\xi_{j}$ for $i=1,\ldots, n, j=1,\ldots,r$.

\subsection{Computation of the weights}\label{weights}
The weight $\omega_{i}, i=1,\ldots,r$ of the decomposition of $\sigma$
can be easily computed using the eigenvectors of all $M_{x_{j}}^{\overline{V}_{r}}, j=1,\ldots,n$ as follows.
\begin{proposition} Let $\sigma=\sum_{i=1}^{r}\omega_i\, \be_{\xi_{i}}$ with $\omega_{i}\in \CC\setminus\{0\}$, $\xi_{i}=(\xi_{i,1},\ldots, \xi_{i,n}) \in \CC^n$ distinct and let $M_{x_{j}}^{\overline{V}_{r}}$ be the matrix of multiplication by $x_j$ in the basis $\overline{V}_{r}$.
Let $\bv_{i}$ be a common eigenvector of $M_{x_{j}}^{\overline{V}_{r}}$, $j=1,.., n$ for the eigenvalues $\xi_{i,j}$. 
Then the weight of $\be_{\xi_{i}}$ in the decomposition of $\sigma$ is 
 \begin{equation}\label{eq:weight}
\omega_{i} = \frac{[1]^{\T}\, H_{\sigma}^{d_{1},d_{2}}\, \overline{V}_{r}\, \bv_{i}}{[\xi_{i}^{\alpha}]_{\alpha\in A_{2}}^{\T} \overline{V}_{r}\, \bv_{i}}.
\end{equation}
\end{proposition}
\begin{proof}
According to Proposition \ref{prop:simpleroot}, the eigenvectors of 
the multiplication operator $\cM_{x_{i}}$ are, up to scalar, the interpolation polynomials $\bu_{i}(\bx)$ at the roots.
Let $\bu_{\xi_{i}}$ be the coefficient vector associated to $\bu_{\xi_{i}}(\bx)$ in the basis $\overline{V}_{r}$ of $\cA_{\sigma}$. 
Let $\bv_{i}=\lambda \bu_{\xi_{i}}$ be the eigenvector of $M_{x_{i}}^{\overline{V}_{r}}$ associated to the eigenvalue $\xi_{j,i}$ for $j=1,\ldots,r, i=1,\ldots,n$ such that $\bv_{i}(\bx)={A_{2}}^{\T}\,\tilde{\bv}_{i}=\sum_{|\beta|\leq d_{2}} \tilde{\bv}_{i\beta} \bx^{\beta}$ where $\tilde{\bv}_{i}=\overline{V}_{r}\,\bv_{i}$. 
Applying the series on all the idempotents, we obtain
$$
\<\sigma\mid \bu_{\xi_{i}}(\bx)\>=\<\sum_{j=1}^{r} \omega_j\be_{\xi_j}(\textbf{y}) \mid \bu_{\xi_{i}}(\bx)\>=\omega_{i}\bu_{\xi_{i}}(\xi_{i})=\omega_{i}.
$$
Therefore, we have $\omega_{i}=\frac{\<\sigma\mid\lambda \bu_{\xi_{i}}(\bx)\>}{\lambda}=\frac{\<\sigma\mid \bv_{i}(\bx)\>}{\lambda}=\frac{\<\sigma\mid \bv_{i}(\bx)\>}{\bv_{i}(\xi_{i})}$
	 because of $\bv_{i}(\xi_{i})=(\lambda \bu_{\xi_{i}})(\xi_{i})=\lambda$.  Then 
$$
<\sigma \mid \bv_{i}(\bx)>=[1]^{\T}\, H_{\sigma}^{d_{1},d_{2}}\tilde{\bv_{i}}=  [1]^{\T}\, H_{\sigma}^{d_{1},d_{2}}\, \overline{V}_{r}\, \bv_{i},
$$ 
where $[1]$ is the vector of coefficients of the polynomial $1$ in the monomial basis $A_{1}=(\bx^{\alpha})_{|\alpha| \leq d_{1}}$ and 
$$
{\bv_{i}(\xi_{i})}=[\xi_{i}^{\alpha}]_{\alpha\in A_{2}}^{\T}\, 
\tilde{\bv}_i = [\xi_{i}^{\alpha}]_{\alpha\in A_{2}}^{\T} \overline{V}_{r}\, \bv_{i}.
$$
We deduce that $\omega_{i} = \frac{[1]^{\T}\, H_{\sigma}^{d_{1},d_{2}}\, \overline{V}_{r}\, \bv_{i}}{[\xi_{i}^{\alpha}]_{\alpha\in A_{2}}^{\T} \overline{V}_{r}\, \bv_{i}}.$
\end{proof}
\subsection{Algorithm}
We describe now the algorithm to recover the sum $\displaystyle\sigma(\by)= \sum_{j=1}^{r} \omega_j\,\be_{\xi_j}(\by)$, $\omega_j \in \CC\setminus \{0\}, \xi_{j} \in \CC^{n}$, from the first coefficients $(\sigma_{\alpha})_{|\alpha| \leq d}$ of the formal power series $\displaystyle\sigma=\sum_{\alpha}\sigma_{\alpha}\frac{\by^{\alpha}}{\alpha!}$.

\begin{algorithm}[H]\caption{Decomposition of polynomial-exponential series with constant weights}\label{algo:genprony1}
\textbf{Input:} the moments $\sigma_{\alpha}$ of $\sigma$ for $|\alpha|\le d$.\\
Let $d_{1}$ and $d_{2}$ be positive integers such that $d_{1}+d_{2}+1 = d$, for example $d_{1}:=\lceil\frac{d-1}{2}\rceil$ and $d_{2}:=\lfloor\frac{d-1}{2}\rfloor$.
	\begin{enumerate}
		\item Compute the Hankel matrix $H_\sigma^{d_1,d_2}=[\sigma_{(\alpha+\beta)}]_{\substack{\mid \alpha \mid \leq d_1 \\ \mid \beta \mid \leq d_2}}$ of $\sigma$ in for the monomial sets $A_{1}=(\bx^{\alpha})_{|\alpha| \leq d_{1}}$ and $A_{2}=(\bx^{\beta})_{|\beta| \leq d_{2}}$.
		\item Compute the singular value decomposition of $H_{\sigma}^{d_{1},d_{2}}=USV^{\T}$ with singular values $s_1\ge s_2\ge \cdots\ge s_m\ge 0$.
		\item Determine its numerical rank, that is, the largest integer $r$ such that $\frac{s_{r}}{s_1} \geq \epsilon$.
		\item Form the matrices  $M_{x_{i}}^{\overline{V}_{r}}=S_{r}^{-1}U_{r}^{\H}H_{x_{i}\star\sigma}^{d_{1},d_{2}}\overline{V}_{r}, i=1,\ldots,n$, where $H_{{x}_{i}\star\sigma}^{d_1,d_2}$ is the Hankel matrix associated to ${x}_{i}\star\sigma$.
    	\item Compute the eigenvectors $\bv_j$ of $\sum_{i}^{n}l_{i}M_{x_{i}}$ for a random choice of $l_i$ in $[-1,1]$, $i=1,\ldots,n$ and for each $j=1,\ldots,r$ do the following:
	    \begin{itemize}
		\item Compute $\xi_{j,i}$ such that $M_i \bv_j = \xi_{j,i} \bv_j$ for $i=1,\ldots, n$ and deduce the point $\xi_j:=(\xi_{j,1},\ldots,\xi_{j,n})$.
		\item Compute $\omega_j = \frac{\< \sigma | \bv_j(\bx)\>}{\bv_j(\xi_j)} = \frac{[1]^{\T}\, H_{\sigma}^{d_{1},d_{2}}\, \overline{V}_{r}\,\bv_j}{[\xi_{i}^{\alpha}]_{\alpha\in A_{2}}^{\T} \overline{V}_{r}\, \bv_{j}}$ where $[1]$ is the coefficient vector of $1$ in the basis $A_{1}$.
	    \end{itemize}
        \end{enumerate}
\textbf{Output:} {$r\in \NN$, $\omega_j\in \CC\backslash{(0)}$, $\xi_{j} \in \mathbb{C}^{n}$, j=$1,\ldots,r$ such that $\sigma(\by)= \sum_{j=1}^{r} \omega_j\,\be_{\xi_j}(\by)$ up to degree $d$.}
\end{algorithm}
\section{Experimentation}\label{sec:5}
In this section, we present numerical experimentations for the decomposition of $\sigma=\sum_{\alpha \in \NN} \sigma_{\alpha}\frac{\by^{\alpha}}{\alpha!}$ from its moments $\sigma_{\alpha}$. For a given number of variables $n$ and a fixed degree $d$, we compute the coefficients   $\sigma_{\alpha}=\sigma(\bx^{\alpha})=\sum_{i=1}^{r} \omega_{i}\xi_{i}^{\alpha}$ such that  $|\sigma_{\alpha}| \leq d$ where $\omega_{j} \in \CC\backslash{(0)}$ and $\xi_{i}=(\xi_{i,1},\ldots,\xi_{i,n}), i=1, \ldots, r$ have random uniform distributions such that $0.5\,M \leq |\xi_{i,j}| \leq 1.5\,M$, $-\pi \leq arg({\xi_{i,j}}) \leq \pi$, $0.5 \leq|\omega_{i}| \leq 1$ and $-\pi \leq arg({\omega_{i}}) \leq \pi$, for $M \geq 1$. To analyse the numerical behaviour of our method, we compare the results with the known frequencies and weights used to compute the moments of $\sigma$.  
 
We use Maple 16 to implement the algorithms. The arithmetic operations are   performed on complex numbers, with a numerical precision fixed to $Digits=15$.
 
\subsection{Numerical behavior against perturbation}
We apply random perturbations on the moments of the form $\sigma_{\alpha}+\epsilon(p_{\alpha}+\boldsymbol{i}\,q_{\alpha})$ where $p_{\alpha}$ and $q_{\alpha}$ are two random numbers in $[-1,1]$ with a uniform distribution, and $\varepsilon=10^{-e}$ where $e$ is a fixed positive integer.
 
To measure the consistency of our algorithm, we compute the maximum error between the input frequencies $\xi_{i}$ and the output frequencies  $\xi_{i}^{'}$, and between the input weights $\omega_{i}$ and the output weights $\omega_{i}^{'}$:
 \begin{equation}\label{eqn:maximumerreur}
\mathrm{err} = \max(err(\xi_{i},\xi_{i}^{'}),\mathrm{err}(\omega_{i},\omega_{i}^{'})) \hspace{1mm} \mathrm{where} \hspace{1mm}     \mathrm{err}(\omega_{i},\omega_{i}^{'})=\max_{1 \leq i \leq r}|\omega_{i}-\omega_{i}^{'}|\hspace{1mm}\mathrm{and}\hspace{1mm} \mathrm{err}(\xi_{i},\xi_{i}^{'})=\max_{1 \leq i \leq r}{\|\xi_{i}-\xi_{i}^{'}\|}_{2}.
 \end{equation}
 In each computation, we compute the average of the maximum errors resulting from $10$ tests. 
 \begin{figure}[ht!]
 	\begin{center}
 		\subfloat[]{%
 			\label{fig:first}
 			\includegraphics[height=0.32\textwidth]{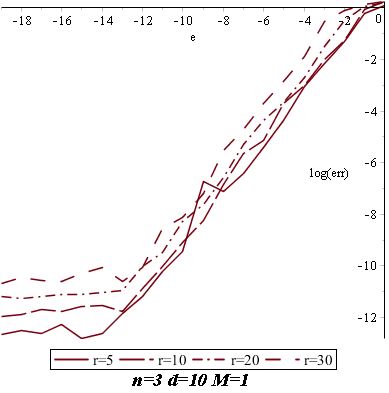}
 		}%
 		\subfloat[]{%
 			\label{fig:second}
 			\includegraphics[height=0.32\textwidth]{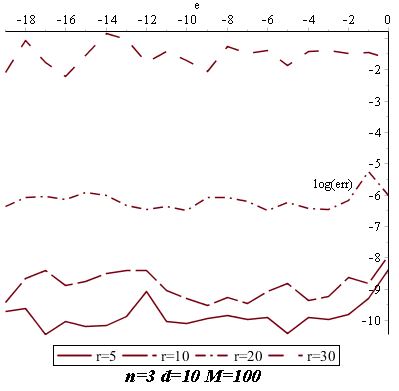}
 		}
 	\end{center}
 	\caption{%
 		The influence of the amplitude of the frequencies on the maximum error.
 	}%
 \end{figure}
 \begin{figure}[ht!]
 	\begin{center}
 		\subfloat[]{%
 			\label{fig:fourth}
 			\includegraphics[height=0.32\textwidth]{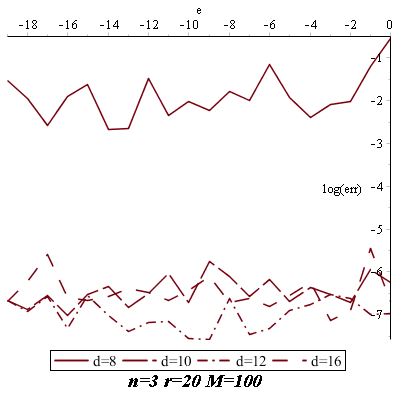}
 		}%
 		\subfloat[]{%
 			\label{fig:fifth}
 			\includegraphics[height=0.32\textwidth]{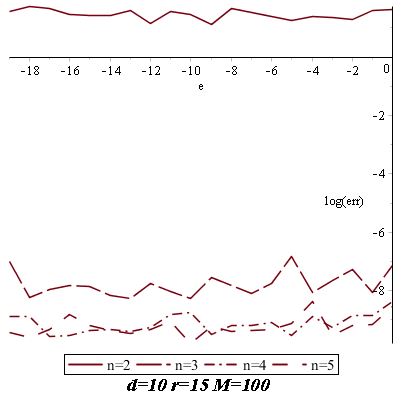}
 		}%
 	\end{center}
 	\caption{The influence of the degree and dimension on the maximum error.}%
 \end{figure}
 
 In Figures \ref{fig:first} and \ref{fig:second}, we study the evolution of the error in terms of the perturbation $\varepsilon=10^{(-e)}$, for a fixed degree $d=10$, a number of variables $n=3$, different ranks $r=5, 10, 20, 30$ and for two different amplitudes of the frequencies $M=1$ and $M=100$.
 
In Figure \ref{fig:first} for $M=1$, the lower error is for the lower rank $r=5$. Between $\varepsilon \approx 10^{-12}$ and $\varepsilon=1$, the error err increases in terms of the perturbation
as  err$=exp(t\, e)$ for some slope $t\approx 1$. 
The slope $t$ remains approximately constant but the error increases slightly with the rank $r$.
Before $\varepsilon = 10^{-13}$, it is approximately constant (approximately $10^{-12}$ for $r=5$) 
This is due to the fact, in this range, the perturbation is lower than the numerical precision.  
 
 In Figure \ref{fig:second} for $M=100$, the lower error is also for the lower rank. The error has almost a constant value when $e$ varies. 
It is bigger than for $M=1$ for small perturbations. For $r=5, 10$ the error slightly increases between $e=-2$ and $e=0$, with a similar slope. This figure clearly shows that the error degrades significantly from $M=1$ to $M=100$ and that the degradation increases rapidly with the rank $r$.
 
In Figure \ref{fig:fourth}, we fix the number of variables $n=3$, the rank $r=20$ and we change the degree $d$ which induces a change in the dimensions of the Hankel matrices. For $e \in [-19,0]$, the error decreases when we increase the degree from $d=8$ to $d=10$. It is slightly lower when $d=12$ than when $d=10$, and error is similar for $d=10$ and $d=16$. This increase of the precision with the degree can be related to ratio of number of moments by the number of values to recover in the decomposition. 
 
 In Figure \ref{fig:fifth}, we fix the degree $d=10$ and the rank $r=15$ and we change the number of variables $n=2, 3, 4, 5$. The dimension of the matrices increases polynomially with $n$. We observe that the error decreases quickly with $n$. It shows that the precision improves significantly with the dimension.

\subsection{Numerical Rank}\label{numericalrank}
 To compute the number $r$ of terms in the decomposition of $\sigma$, we arrange the diagonal entries in the decreasing order 
 $s_{1} \geq s_{2} \geq \cdots \geq s_{r} > s_{r+1} \geq \cdots \geq {s_{N_2}}$ and we determine the numerical rank of $H_\sigma^{d_1,d_2}$ by fixing the largest integer $r$ such that $s_{r}/{s_{1}} \geq \epsilon$. 
 
  It is known that the ill-conditioning of the Hankel matrix associated to Prony's method is in the origin of a numerical instability with respect to perturbed measurements
$$\tilde{\sigma}_{\alpha+\beta}=\sigma_{\alpha+\beta}+\varepsilon_{\alpha+\beta}\qquad |\alpha+\beta|\leq d.$$
In our algorithm the computation of the numerical rank can be affected by this instability. We can explain this instability, using a reasoning close to \cite{sauer_pronys_2016-1}, as follows.

We denote by $s_j$ (resp. $\tilde{s}_j$) the $j^{\mathrm{th}}$ largest singular value of $H:=H_\sigma^{d_1,d_2}$ (resp. $\tilde{H}:=H_{\tilde{\sigma}}^{d_1,d_2}$). The perturbation result for singular values satisfies the estimate (see \cite{golub_matrix_1996})
$$|s_j-\tilde{s}_j|\leq s_1(\varepsilon)=\|\varepsilon\|_2.$$
Then, as long as the perturbation is small relative to the conditioning of the problem, that is
$$\|\epsilon\|_2\leq \tfrac{1}{2}s_r\quad\text{ provided that } r=\text{rank}(H),$$
then $\|s_j-\tilde{s}_j\|\leq\frac{1}{2}s_r\;\forall j$ and therefore $\tilde{s}_r\geq\frac{1}{2}s_r$ and $\tilde{s}_{r+1}\leq\frac{1}{2}s_r$. Hence by taking $\varepsilon\leq\frac{1}{2}s_r$ as a threshold level we will be sure that the rank is calculated correctly.

But the problem may be badly ill-conditioned and then such a level will not be reasonable. In fact
$$H=(\sigma_{\alpha+\beta})_{\substack{|\alpha|\leq d_1\\|\beta|\leq d_2}}=\left(\sum_{i=1}^r\omega_i\xi_i^{\alpha+\beta}\right)_{\substack{|\alpha|\leq d_1\\|\beta|\leq d_2}}=\sum_{i=1}^r\omega_iv_{i,d_1}v_{i,d_2}^T,$$
where $v_{i,d_1}=(\xi_i^{|\alpha|})_{|\alpha|\leq d_1}$ (resp. $v_{i,d_2}=(\xi_i^{|\beta|})_{|\beta|\leq d_2}$) is the $i^{\mathrm{th}}$  column of the Vandermonde matrix $V_{d_1}=(\xi_i^\alpha)_{\substack{1\leq i\leq r\\ |\alpha|\leq d_1}}$ (resp. $V_{d_2}=(\xi_i^\beta)_{\substack{1\leq i\leq r\\ |\beta|\leq d_2}}$). Then 
$$H=\sum_{i=1}^r\omega_i V_{d_1}e_ie_i^TV_{d_2}^T=V_{d_1}\left(\sum_{i=1}^r\omega_ie_ie_i^T\right)V_{d_2}^T=V_{d_1}CV_{d_2}^T$$
where $C=\text{diag}\left((\omega_i)_{1\leq i\leq r}\right)$ is the diagonal matrix with $\omega_i$ on the diagonal.

Now, using the fact that
$$s_r(H)=\min_{\substack{\|x\|=1\\ Hx\neq0}}\|Hx\|_2=\min_{\substack{\|x\|=1\\ Hx\neq0}}\|V_{d_1}CV_{d_2}^Tx\|_2,$$
we remark that if $V_{d_2}$ (resp. $V_{d_1}$) is ill-conditioned then $\|V_{d_2}x\|_2$ (resp. $\|V_{d_1}CV_{d_2}^Tx\|$) may be very small and $s_r(H)$ is small as well. This situation can also be produced if $\max_{1\leq i\leq r}\omega_i$ is very small.
In our numerical experiments, the $\omega_i$ are chosen randomly in $[0.5,1]$ and then they don't cause any numerical instability.

On the other hand, the $\xi_i$ vary in such a way that their amplitude can be large, which can generate very ill-conditioned Vandermonde matrices. In fact, it is known (see \cite{pan_how_2016}), that for a nonsingular univariate Vandermonde matrix $V=(a_i^j)_{0\leq i,j\leq n-1}$, where $(a_i)_{0\leq i\leq n-1}$ denotes a vector of $n$ distinct knots, the condition number of $V$ is exponential in $n$ if $\max_{0\leq i\leq n-1}|a_i|>1$ or in $k$ if $|a_i|<1$ for at least $k$ knots $a_i$. Therefore an $n\times n$ Vandermonde matrix is badly ill-conditioned unless all knots lie in or near the disc $D(0,1)=\{x:\;|x|\leq1\}$ and unless they lie mostly on or near its boundary $C(0,1)$.

In the multivariate case, it appears that the condition number of multivariate Vandermonde matrices has the same behavior as in univariate case. That is, it is exponential in the highest degree of the entries.

According to the foregoing, when the amplitude $M$ of the frequencies increases (even for moderate values of $M$) the numerical rank calculated by truncating the singular values of $H$ will be different from the exact rank of $H$. An idea to remedy this problem is to rescale the frequencies $\xi_i$ in order to obtain points with coordinates close to the unitary circle $C(0,1)$.
 
\subsection{Rescaling}
As we have seen in Figures \ref{fig:first} and \ref{fig:second}, the error increases significantly with the amplitude $M$. To remedy this issue, we present a rescaling technique and its numerical impact. It's done like this:
 \begin{itemize} 
 	\item For a chosen non-zero constant $\lambda$, we transform the input moments of the series as follows:
 	$$\sigma(\by):=\sum_{\alpha \in N^{n}}\sigma_{\alpha}\frac{\textbf{y}^{\alpha}}{\alpha!} \longrightarrow \tilde{\sigma}(\by):=\sigma(\lambda\by)=\sum_{\alpha \in N^{n}}\lambda^{|\alpha|}\sigma_{\alpha}\frac{\textbf{y}^{\alpha}}{\alpha!},$$ 
 	which corresponds to the scaling on the frequencies $e_{\xi}(\lambda\by)=e_{\lambda\xi}(\by)$.
 	\item We compute decomposition of $\tilde{\sigma}(\by)=\sigma(\lambda\by)$ from the moments $\tilde{\sigma}_{\alpha}=\lambda^{|\alpha|}\sigma_{\alpha}$.
 	\item We apply the inverse scaling on the computed frequencies $\tilde{\xi}_{i}$ which gives $\xi_{i}=\frac{\tilde{\xi}_{i}}{\lambda}=\big(\frac{\tilde{\xi}_{i,1}}{\lambda},\ldots , \frac{\tilde{\xi}_{i,n}}{\lambda} \big)$.
 	\end{itemize}
 
 	To determine the scaling factor $\lambda$, we use $\lambda:=\frac{1}{m}$ where $m=\frac{max_{|\alpha|=d}|\sigma_{\alpha}|}{max_{|\alpha|={d-1}} |\sigma_{\alpha}|}$. This is justified as follows: If $|\omega_{j}| \leq 1, j=1,\dots,r$, then $|\sigma_{\alpha}|=|\sum_{j=1}^{r} \omega_{j}{\xi_{j}}^{\alpha}| \simeq M^{d}$ for $|\alpha|=d$ big and for $M$ is the highest modulus of frequencies. Similarly  $|\sigma_{\alpha^{'}}| \simeq M^{d-1}$ for $|\alpha^{'}|=d-1$. Then we have  $m=\frac{max_{|\alpha|=d}|\sigma_{\alpha}|}{max_{|\alpha|={d-1}} |\sigma_{\alpha}|} \approx M$.
 \begin{figure}[ht!]
 	\begin{center}
 			\includegraphics[height=0.32\textwidth]{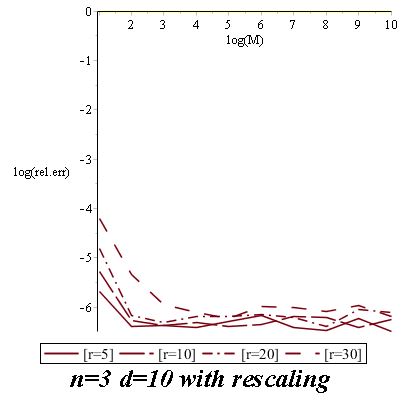}
 			\caption{The rescaling influence}
 			\label{fig:three}	
 	\end{center}
 \end{figure}

To study the numerical influence of the rescaling, we compute the maximum relative error between the input frequencies $\xi_{i}$ and the output frequencies  $\tilde{\xi}_{i}$, and the maximum error between the input weights $\omega_{i}$ and the output weights $\tilde{\omega}_{i}$, and we take their maximum:
\begin{equation}
\mathrm{rel.err} = \mathrm{\max(rel.err}(\xi_{i},\tilde{\xi}_{i}),\mathrm{err}(\omega_{i},\tilde{\omega}_{i}))
\end{equation}
 $\mathrm{where}\hspace{1mm}\mathrm{err}(\omega_{i},\tilde{\omega}_{i})=\max_{1 \leq i \leq r}|\omega_{i}-\tilde{\omega}_{i}| \hspace{1mm}\mathrm{and}\hspace{1mm} \mathrm{rel.err}(\xi_{i},\tilde{\xi}_{i})=\max_{1 \leq i \leq r}\frac{{\|\xi_{i}-\tilde{\xi}_{i}\|}_{2}}{{\|\xi_{i}\|}_{2}}$.
 	
 In Figure \ref{fig:three}, we see the influence of the rescaling on the maximum relative error. The perturbation on the moments is of the order 
$\varepsilon=10^{-6}$. Each curve for $r=5, 10, 20, 30,$ has almost a constant evolution with the increasing values of $M$ between $10^{2}$ and $10^{10}$. The maximum relative error is lower when $M=100$ than when $M=1$ which is confirmed with the results shown in Figures \ref{fig:first} and \ref{fig:second}. When we increase $r$ the maximum relative error decreases slightly. 

In conclusion, the rescaling has an important influence on the computation of the maximum relative error when the modulus $M$ of points is quite big.

The scaling of moments by some computed factor $\lambda$ also enhances the computation of the numerical rank $r$ and leads to a better decomposition as we have seen in \ref{numericalrank}.

\subsection{Newton iteration}
Given a perturbation $\tilde{\sigma}=\sum_{\alpha}\tilde{\sigma}_{\alpha} \frac{\by^{\alpha}}{\alpha!}$ of a polynomial-exponential series $\sigma=\sum_{i=1}^{r}\omega_{i} \be_{\xi_{i}}(\by)$, we want to remove the perturbation on $\tilde{\sigma}$ by computing the polynomial-exponential series of rank $r$, which is the closest to the perturbed series $\tilde{\sigma}$. Starting from an approximate decomposition, using the previous method on the perturbed data, we apply a Newton-type method to minimize the distance between the input series and a weighted sum of $r$ exponential terms.

To evaluate the distance between the series, we use the first moments $\tilde{\sigma}_{\alpha}$ for $\alpha\in A$, where $A$ is a finite subset of $\NN^{n}$.
For $\alpha \in A$, let $F_{\alpha}(\Xi) = \sum_{i=1} \omega_{i} \xi_{i}^{\alpha}- \tilde{\sigma}_{\alpha}$ be the error function for the moment $\tilde{\sigma}_{\alpha}$, where $\omega_{i}, \xi_{i,j}$ are variables.
We denote by $\Xi=(\xi_{i,j})_{1\le i\le r, 0 \le j\le n }$ this set of variables, with the convention that $\xi_{i,0}= \omega_{i}$ for $i= 1, \ldots, r$. Let $I=[1,r]\times [0,n]=\{(i,j)\mid 1\le i\le r, 0 \le j\le n \}$ be the indices of the variables and $N=(n+1)\,r=|I|$.
We denote by $F(\Xi)=(F_{\alpha}(\Xi))_{\alpha\in A}$ the vector of these error functions.

We want to minimize the distance
$$ 
E(\Xi)=\frac{1}{2}\sum_{\alpha\in A} |F_{\alpha}(\Xi)|^{2} =\frac{1}{2} \| F(\Xi) \|^{2}.
$$

Let $M(\Xi_{i})=[w_{i} \xi_{i}^{\alpha}]_{\alpha\in A}$.
We denote by $V(\Xi)=\left( \partial_{(i,j)} M(\Xi_{i})\right)_{(i,j)\in I}$ the $|A|\times N$ Vandermonde-like matrix, which columns are the vectors $\partial_{(i,j)} M(\Xi_{i})$. 
The gradient of $E(\Xi)$ is
$$ 
\nabla E(\Xi) = ( \< \partial_{{(i,j)}} M(\Xi_{i}), F(\Xi)\>)_{(i,j)\in I} = V(\Xi)^{T} F(\Xi)
$$
where $\partial_{(i,j)}$ is the derivation with respect to $\Xi_{i,j}$ for $(i,j)\in I$. 
We denote by $V(\Xi)=\left( \partial_{(i,j)} M(\Xi_{i})\right)_{(i,j)\in I}$ the $|A|\times N$ Vandermonde-like matrix, which columns are the vectors $\partial_{(i,j)} M(\Xi_{i}), (i,j) \in I$. 

To find a local minimizer of $E(\Xi)$, we compute a solution of the system $\nabla E(\Xi)=0$, by Newton method.
The Jacobian of $\nabla E(\Xi)$ with respect to the variables $\Xi$ is 
\begin{align*}
J_{\Xi}(\nabla E) &= 
\< \partial_{{(i,j)}} M(\Xi_{j}) , \partial_{(i',j')} M(\Xi_{j'})\>
+ \left( \< \partial_{{(i,j)}} \partial_{(i',j')} M(\Xi_{i}), F(\Xi)\>\right)_{(i,j)\in I,(i',j')\in I} \\
& =  V(\Xi)^{\T} V(\Xi) + \left( \< \partial_{{(i,j)}} \partial_{(i',j')} M(\Xi_{i}), F(\Xi)\> \right)_{(i,j)\in I,(i',j')\in I}.
\end{align*}
Notice that $\partial_{{(i,j)}} \partial_{(i',j')} M(\Xi_{i})=0$ if $i\neq i'$ so that the first matrix is a block diagonal matrix.
Then, Newton iteration takes the form:
$$ 
\Xi_{n+1} = \Xi_{n} - J_{\Xi}(\nabla E)^{-1} \nabla E(\Xi_{n}).
$$

To study the numerical influence of Newton method, we compute the maximum absolute error between the input frequencies $\tilde{\xi}_{i}$ and the output frequencies $\tilde{\xi}_{i}$, and the maximum error between the input weights $\omega_{i}$ and the output weights $\tilde{\omega}_{i}$ as in \eqref{eqn:maximumerreur}.

\begin{figure}[ht!]
	\begin{center}
		\subfloat[]{%
			\label{fig:six}
			\includegraphics[height=0.32\textwidth]{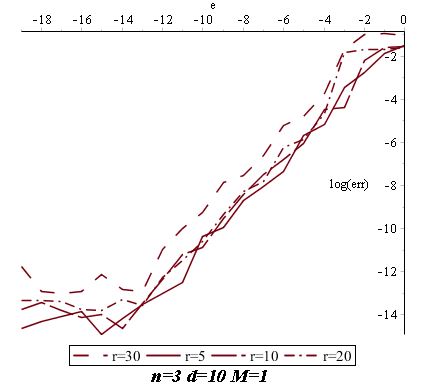}
		}%
		\subfloat[]{%
			\label{fig:seventh}
			\includegraphics[height=0.32\textwidth]{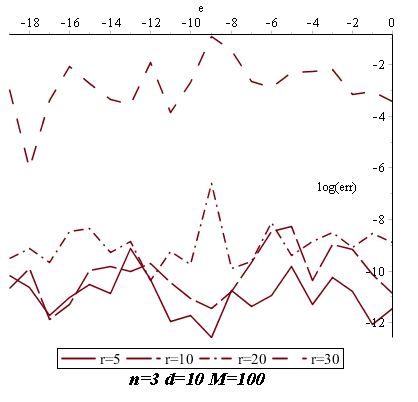}
		}
	\end{center}
	\caption{Newton influence with 5 iterations}
\end{figure}

Figures \ref{fig:six} and \ref{fig:seventh} show that Newton iterations improve the error.  
The error decreases by a factor of $\approx 10^{2}$ compared to the computation without Newton iterations. 
In Figure \ref{fig:seventh} for $M=100$ the error is smaller than without Newton iterations by a similar order of magnitude (see in Figure \ref{fig:second}).

\ \\\noindent{}\textbf{References}

\begin{thebibliography}{BCMT10}

\bibitem[ACd10]{andersson_nonlinear_2010}
Fredrik Andersson, Marcus Carlsson, and Maarten~V. {de Hoop}.
\newblock Nonlinear approximation of functions in two dimensions by sums of
  exponential functions.
\newblock {\em Applied and Computational Harmonic Analysis}, 29(2):156--181,
  September 2010.

\bibitem[{Bar}95]{baron_de_prony_essai_1795}
Gaspard~Riche {Baron de Prony}.
\newblock Essai exp{\'e}rimental et analytique: Sur les lois de la
  dilatabilit{\'e} de fluides {\'e}lastique et sur celles de la force expansive
  de la vapeur de l'alcool, {\`a} diff{\'e}rentes temp{\'e}ratures.
\newblock {\em J. Ecole Polyt.}, 1:24--76, 1795.

\bibitem[BBCM13]{bernardi_general_2013}
Alessandra Bernardi, J{\'e}r{\^o}me Brachat, Pierre Comon, and Bernard
  Mourrain.
\newblock General tensor decomposition, moment matrices and applications.
\newblock {\em Journal of Symbolic Computation}, 52:51--71, 2013.

\bibitem[BCMT10]{brachat_symmetric_2010}
J{\'e}r{\^o}me Brachat, Pierre Comon, Bernard Mourrain, and Elias~P.
  Tsigaridas.
\newblock Symmetric tensor decomposition.
\newblock {\em Linear Algebra and Applications}, 433(11-12):1851--1872, 2010.

\bibitem[Bec97]{beckermann_condition_1997}
Bernhard Beckermann.
\newblock The {{Condition Number}} of real {{Vandermonde}}, {{Krylov}} and
  positive definite {{Hankel}} matrices.
\newblock {\em Numerische Mathematik}, 85:553--577, 1997.

\bibitem[BGL07]{beckermann_numerical_2007}
Bernhard Beckermann, Gene~H. Golub, and George Labahn.
\newblock On the numerical condition of a generalized {{Hankel}} eigenvalue
  problem.
\newblock {\em Numerische Mathematik}, 106(1):41--68, 2007.

\bibitem[BM05]{beylkin_approximation_2005}
Gregory Beylkin and Lucas Monz{\'o}n.
\newblock On approximation of functions by exponential sums.
\newblock {\em Applied and Computational Harmonic Analysis}, 19(1):17--48, July
  2005.

\bibitem[BP94]{bini_polynomial_1994}
Dario Bini and Victor~Y. Pan.
\newblock {\em Polynomial and {{Matrix Computations}}}.
\newblock {Birkh{\"a}user Boston}, Boston, MA, 1994.

\bibitem[EM07]{elkadi_introduction_2007}
Mohamed Elkadi and Bernard Mourrain.
\newblock {\em {Introduction {\`a} la r{\'e}solution des syst{\`e}mes
  polynomiaux}}, volume~59 of {\em Math{\'e}matiques et Applications}.
\newblock {Springer}, 2007.

\bibitem[Fuh12]{fuhrmann_polynomial_2012}
Paul~A. Fuhrmann.
\newblock {\em A {{Polynomial Approach}} to {{Linear Algebra}}}.
\newblock Universitext. {Springer New York}, New York, NY, 2012.

\bibitem[GL96]{golub_matrix_1996}
Gene~H. Golub and Charles F.~Van Loan.
\newblock {\em Matrix {{Computations}}}.
\newblock {JHU Press}, 1996.

\bibitem[GP03]{golub_separable_2003}
Gene Golub and Victor Pereyra.
\newblock Separable nonlinear least squares: The variable projection method and
  its applications.
\newblock {\em Inverse Problems}, 19(2):R1--R26, 2003.

\bibitem[HS90]{hua_matrix_1990}
Yingbo Hua and Tapan~K. Sarkar.
\newblock Matrix pencil method for estimating parameters of exponentially
  damped/undamped sinusoids in noise.
\newblock {\em IEEE Transactions on Acoustics, Speech, and Signal Processing},
  38(5):814--824, 1990.

\bibitem[KPRv16]{kunis_multivariate_2016}
Stefan Kunis, Thomas Peter, Tim R{\"o}mer, and Ulrich {von der Ohe}.
\newblock A multivariate generalization of {{Prony}}'s method.
\newblock {\em Linear Algebra and its Applications}, 490:31--47, February 2016.

\bibitem[Kro81]{kronecker_zur_1881}
Leopold Kronecker.
\newblock Zur {{Theorie}} der {{Elimination Einer Variabeln}} aus {{Zwei
  Algebraischen Gleichungen}}.
\newblock pages 535--600., 1881.

\bibitem[Lan11]{landsberg_tensors:_2011}
Joseph~M. Landsberg.
\newblock {\em Tensors: {{Geometry}} and {{Applications}}}.
\newblock Graduate studies in mathematics. {American Mathematical Soc.}, 2011.

\bibitem[Mar12]{markovsky_low_2012}
Ivan Markovsky.
\newblock {\em Low {{Rank Approximation}}}.
\newblock Communications and Control Engineering. {Springer London}, London,
  2012.

\bibitem[Mou96]{mourrain_isolated_1996}
Bernard Mourrain.
\newblock Isolated points, duality and residues.
\newblock {\em J. of Pure and Applied Algebra}, 117\&118:469--493, 1996.

\bibitem[Mou16]{mourrain_polynomial-exponential_2016}
Bernard Mourrain.
\newblock Polynomial-exponential decomposition from moments, 2016.
\newblock hal-01367730, arXiv:1609.05720.

\bibitem[MP00]{mourrain_multivariate_2000}
Bernard Mourrain and Victor~Y. Pan.
\newblock Multivariate {{Polynomials}}, {{Duality}}, and {{Structured
  Matrices}}.
\newblock {\em Journal of Complexity}, 16(1):110--180, March 2000.

\bibitem[Pan16]{pan_how_2016}
Victor~Y. Pan.
\newblock How {{Bad Are Vandermonde Matrices}}?
\newblock {\em SIAM Journal on Matrix Analysis and Applications},
  37(2):676--694, January 2016.

\bibitem[PPS15]{peter_pronys_2015}
Thomas Peter, Gerlind Plonka, and Robert Schaback.
\newblock Prony's {{Method}} for {{Multivariate Signals}}.
\newblock {\em PAMM}, 15(1):665--666, 2015.

\bibitem[PS12]{pereyra_exponential_2012}
V.~Pereyra and G.~Scherer, editors.
\newblock {\em Exponential {{Data Fitting}} and Its {{Applications}}}.
\newblock {Bentham Science Publisher}, 2012.

\bibitem[PT11]{potts_nonlinear_2011}
Daniel Potts and Manfred Tasche.
\newblock Nonlinear approximation by sums of nonincreasing exponentials.
\newblock {\em Applicable Analysis}, 90(3-4):609--626, 2011.

\bibitem[PT13]{potts_parameter_2013}
Daniel Potts and Manfred Tasche.
\newblock Parameter estimation for multivariate exponential sums.
\newblock {\em Electronic Transactions on Numerical Analysis}, 40:204--224,
  2013.

\bibitem[RK89]{roy_esprit-estimation_1989}
Richard Roy and Thomas Kailath.
\newblock {{ESPRIT}}-estimation of signal parameters via rotational invariance
  techniques.
\newblock {\em IEEE Transactions on Acoustics, Speech, and Signal Processing},
  37(7):984--995, 1989.

\bibitem[Sau16]{sauer_pronys_2016-1}
Tomas Sauer.
\newblock Prony's method in several variables.
\newblock {\em Numerische Mathematik}, 2016.
\newblock To appear.

\bibitem[SK92]{swindlehurst_performance_1992}
A.~Lee Swindlehurst and Thomas Kailath.
\newblock A performance analysis of subspace-based methods in the presence of
  model errors. {{I}}. {{The MUSIC}} algorithm.
\newblock {\em IEEE Transactions on signal processing}, 40(7):1758--1774, 1992.

\bibitem[Tyr94]{tyrtyshnikov_how_1994}
Evgenij~E. Tyrtyshnikov.
\newblock How bad are {{Hankel}} matrices?
\newblock {\em Numerische Mathematik}, 67(2):261--269, 1994.

\end{thebibliography}

\end{document}